\documentclass{article}
\usepackage{amsthm}
\usepackage{amsmath}
\usepackage{amsfonts}
\usepackage[letterpaper,body={13.5cm,22cm}, mag=1000]{geometry}
\usepackage{amssymb}
\usepackage{secdot}
\usepackage{setspace}
\usepackage{titlesec}
\titleformat{\section}[runin]{\bfseries\filcenter}{\thesection}{1em}{}

\renewcommand{\thesection}{\arabic{section}}
\title{\large \bf Finite $p$-groups with minimum number of central automorphisms fixing the center element-wise}
\author{\small \bf Deepak Gumber\footnote{Research supported by NBHM, Department of Atomic Energy.} \\
\small \em School of Mathematics and Computer Applications \\
\small \em Thapar University, Patiala, India\\
\\
\small \bf Mahak Sharma\\
\small \em Department of Applied Sciences and Humanities\\
\small \em Baddi University of Emerging Sciences and Technology, Baddi, India \\
} 
\date{}
\DeclareMathOperator{\inn}{Inn}
\DeclareMathOperator{\Hom}{Hom}

\DeclareMathOperator{\aut}{Aut}
\DeclareMathOperator{\cent}{Aut_z}

\DeclareMathOperator{\centz}{Aut_z^{z}}

\newtheorem{thm}{Theorem}[section]
\newtheorem{lm}[thm]{Lemma}

\begin{document}
\maketitle
\begin{abstract}
\noindent {\bf Abstract.}
We characterize finite $p$-groups $G$ of order upto $p^7$ for which the group of central automorphisms fixing the center element-wise is of minimum possibe order.
\end{abstract}

\vspace{2ex}


\vspace{2ex}


\section{Introduction} 
Let $G$ be a finite group and let $G'$ and $Z(G)$ respectively denote the commutator subgroup and the center of $G$. An automorphism $\alpha$ of $G$ is called a central automorphism if it commutes with all inner automorphisms; or equivalently $x^{-1}\alpha(x)\in Z(G)$ for all $x\in G$. A central automorphism, by definition, fixes $G'$ element-wise. By $\cent(G)$ we denote the group of all central automorphisms of $G$, and by $\centz(G)$ we denote the group of those central automorphisms of $G$ which fix $Z(G)$ element-wise. The central automorphism group can be as large as possible when all automorphisms are central, that is, when $\cent(G)=\aut(G)$; and can be as small as possible when $\cent(G)=Z(\inn(G))$. If $G$ is abelian, then $\inn(G)$ is trivial and hence $\cent(G)=\aut(G)$. If $G$ is non-abelian and if $\cent(G)=\aut(G)$, then $\inn(G)$ is abelian and hence $G$ is a nilpotent group of class 2. So one can restrict attention to $p$-groups. Non-abelian $p$-groups for which all automorphisms are central have been well studied. If $\aut(G)$ is abelian, then necessarily $\cent(G)=\aut(G)$, and various authors have considered this situation; and if $\aut(G)$ is non-abelian, even then all automorphisms may be central and this case has also been well explored. Curran \cite{cur} was the first to consider the opposite extreme$-$the case when $\cent(G)$ is minimum possible. He proved that for $\cent(G)$ to be equal to $Z(\inn(G))$, it is necessary that $Z(G)$ must be contained in $G'$ and $Z(\inn(G))$ must not be cyclic \cite[Corollaries 3.7, 3.8]{cur}. Sharma and Gumber \cite{shagum} and recently Kalra and Gumber \cite{kalgum} have characterized finite $p$-groups $G$ of order upto $p^7$ for which $\cent(G)=Z(\inn(G))$. The purpose of this note is to find finite $p$-groups $G$ for which 

\begin{equation}
Z(\inn(G))=\centz(G)<\cent(G).
\end{equation}
In Theorem 2.3, we prove that there is no finite $p$-group $G$ of order upto $p^6$ satisfying (1); and in Theorem 2.4, we prove that a group $G$ of order $p^7$ satisfies (1) if and only if $Z(G)\simeq C_{p^2}$, $|G'|=p^4$ and nilpotence class of $G$ is 4.

It follows from the results of Curran \cite[Theorem 2.3]{cur} and Adney and Yen \cite[Theorem 1]{adnyen} that for any non-abelian group $G$,
$$\centz(G)\simeq \Hom(G/G'Z(G),Z(G)),$$
and if $G$ is purely non-abelian and finite, then 
$$|\cent(G)|=|\Hom(G/G',Z(G))|.$$
We shall need these results quite frequently in our proofs.

\section{Proofs of Theorems}
Let $G$ be a finite $p$-group. Observe that if $Z(G)\le G'$, then $\centz(G)=\cent(G)$. Therefore, hereafter, we shall assume that $Z(G)$ is not contained in $G'$. Attar \cite{att} proved that $\centz(G)=\inn(G)$ if and only if $G'=Z(G)$ and $Z(G)$ is cyclic. If $G$ is nilpotent of class 2, then $\inn(G)=Z(\inn(G))$; it follows that if $\centz(G)=Z(\inn(G))$, then $G'=Z(G)$, which is a contradiction to our first assumption. It is, therefore, also necessary to assume that $G$ is of nilpotence  class at least 3. With these two assumptions, one can observe that $G$ is not of maximal class, $|Z(G)|\ge p^2$ and $|G|$ is at least $p^5$. We start with the following lemma.

\begin{lm}
If $G'$ is abelian, then $G/G'Z(G)$ is not cyclic.
\end{lm}
\begin{proof}
If $G/G'Z(G)=\langle aG'Z(G)\rangle$ is cyclic, then $G=\langle a,G'Z(G)\rangle$. It is easy to see that $G'=\{[g,a]|g\in G'Z(G)\}$ and the map $g\mapsto [g,a]$ is a homomorphism of $G'Z(G)$ onto $G'$ with kernel $Z(G)$. It follows that $|G'Z(G)|=|G'||Z(G)|$, which is not possible in a finite $p$-group.
\end{proof}

\begin{lm}
If $G$ is of coclass $2$, then $G$ does not satisfy $(1)$.
\end{lm}
\begin{proof}
Let $|G|=p^n$. Then $|Z(G)|=p^2$ and $G/Z(G)$ is a maximal class group of order $p^{n-2}$. It follows that $|Z(\inn(G))|=p$, $|(G/Z(G))'|=p^{n-4}$ and hence $|\centz(G)|=|\Hom(G/G'Z(G),Z(G))|\ge p^2>|Z(\inn(G))|.$   
\end{proof}

\begin{thm}
There does not exist a group of order upto $p^6$ satisfying $(1)$.
\end{thm}
\begin{proof} The result follows from Lemma 2.2 if $|G|=p^5$. Therefore suppose that $|G|=p^6$. Observe that either (i) nilpotence class of $G$ is 4 and $|Z(G)|=p^2$ or (ii) nilpotence class of $G$ is 3 and  $|Z(G)|=p^2$ or  $p^3$. The result follows again from Lemma 2.2 in case (i). Assume that nilpotence class of $G$ is 3 and $|Z(G)|=p^3$. Then $G/Z(G)$ is an extra-special $p$-group and hence $|Z(\inn(G))|=|(G/Z(G))'|=p$. It follows that $|G'Z(G)|=p^3$ and hence
$$|\centz(G)|=|\Hom(G/G'Z(G),Z(G))|\ge p^2>|Z(\inn(G))|.$$
Finally assume that nilpotence class of $G$ is 3 and $|Z(G)|=p^2$. Then $G/Z(G)$ is a class 2 group of order $p^4$; consequently $|Z(\inn(G))|=p^2$, $|(G/Z(G))'|=p$ and hence $G'Z(G)$ is an abelian normal subgroup of $G$ of order $p^3$. It follows from Lemma 2.1 that $G/G'Z(G)$ is not cyclic and hence 
$|\centz(G)|>p^2>|Z(\inn(G))|.$
\end{proof}

\begin{thm}
A group $G$ of order $p^7$ satisfies $(1)$ if and only if $Z(G)\simeq C_{p^2}$, $|G'|=p^4$ and nilpotence class of $G$ is $4$.
\end{thm}
\begin{proof}
First suppose that $Z(G)\simeq C_{p^2}$, $|G'|=p^4$ and nilpotence class of $G$ is $4$. Then $G$ is purely non-abelian and hence
$$|\centz(G)|=|\Hom(G/G'Z(G),Z(G))|=p^2<p^3=|\Hom(G/G',Z(G))|=|\cent(G)|.$$
Now $G/Z(G)$ is a group of order $p^5$ and of nilpotence class 3. Therefore $|Z(\inn(G))|=p$ or $p^2$. Assume that $|Z(\inn(G))|=p$. Then $|Z_2(G)|=p^3$. Since $G$ is of class 4 and $Z(G)$ is not contained in $G'$, $\gamma_3(G)<Z_2(G)$ and hence $|\gamma_3(G)|=p^2$. Observe that $G/G'$ is isomorphic to $C_{p^2}\times C_p$ or $C_p\times C_p\times C_p$. The first case is not possible by \cite[Theorem 1.5(iii)]{bla}. In the second case, $G'=\Phi(G)$. We can choose a minimal generating set $\{x,y,z\}$ of $G$ in which one of the generators, say $z$, is in $Z(G)$. Then $G'=\langle [x,y],\gamma_3(G)\rangle$, and hence $G'/\gamma_3(G)$ is a cyclic group of order $p^2$. But this is a contradiction to \cite[Theorem 1.5(i)]{bla}. It follows that $|Z(\inn(G))|=p^2$ and hence $G$ satisfies (1).

Conversely suppose that $G$ satisfies (1). In view of Lemma 2.2, nilpotence class of $G$ is either 3 or 4. First assume that $G$ is of nilpotence class 3. Then $G'$ is abelian and therefore
$$d(\centz(G))\ge d(G/G'Z(G))\ge 2$$ 
by Lemma 2.1. It follows that $p^2\le |Z(\inn(G))|\le p^3$ and hence $p^2\le |Z(G)|\le p^3$. If $|Z(G)|=p^3$, then $|Z(\inn(G))|=p^2$, $|(G/Z(G))'|=p$ by \cite[Theorem 2]{kim} and thus $|G'Z(G)|=p^4$. Since $G/G'Z(G)$ is not cyclic, 
$$|\centz(G)|=|\Hom(G/G'Z(G),Z(G))|\ge p^3,$$ 
which is a contradiction to (3.1). We therefore suppose that $|Z(G)|=p^2$.
First consider the case when $|Z(\inn(G))|=p^{2}$. 
Observe that $|G'|< p^4$, because if $|G'|\ge p^4$, then $|Z_{2}(G)|=|Z(G)G'|\ge p^5$ and hence $|Z(\inn(G))|\ge p^3$. If $|G'|=p^{2}$ or $p^3$, then 
$G/G'Z(G)$ is not cyclic by Lemma 2.1.  Thus 
$$|\centz(G)|=|\Hom(G/G'Z(G),Z(G))|\ge p^3,$$ 
a contradiction to (3.1).  Next consider the case when $|Z(\inn(G))|=p^{3}$. Let $H=G/Z(G)$. Then $|H|=p^5$, $|H/Z(H)|=p^2$ and nilpotence class of $H$ is 2. Thus $|H'|=p$ by \cite[Theorem 2.1(i)]{weigold}. Therefore $|G'Z(G)|=p^3$ and hence $|G'|=p^2$. 
Now
$$|\centz(G)|=|\Hom(G/G'Z(G),Z(G))| = |Z(\inn(G))|=p^3,$$ which is possible only when $Z(G)\simeq C_{p^2}$ and $G/G'Z(G)\simeq C_{p^3}\times C_{p} $. Let 
$$G/G'Z(G)\simeq \langle xG'Z(G) \rangle \times \langle yG'Z(G) \rangle ,$$ 
where $x^{p^3},y^{p} \in G'Z(G)$ but $x^{p},x^{p^2} \notin G'Z(G) $.
  If $G'$ is elementary abelian, then for any $g\in G,$
$$[x^p,g]=[x,g]^{x^{p-1}}[x,g]^{x^{p-2}}\ldots [x,g]^x[x,g]=[x,g]^p[x,g,x]^{p(p-1)/2}=1$$ 
implies that $x^p \in Z(G)$; and if $G'$ is cyclic, then $G$ is regular and thus $1=[x,g]^{p^2}=[x^{p^2},g]$ implies that $x^{p^2} \in Z(G)$. In both the cases, we get a  contradiction. 
 
 We next assume that nilpotence class of $G$ is 4. Observe that $|G'|<p^5$, because if $|G'|=p^5$, then $d(G)=2$ and one of the generators $x,y$ of $G$ will lie in $Z(G)$ and thus $G'=\langle [x,y],\gamma_3(G)\rangle=\gamma_3(G)$, which is not so. We now divide the proof in three cases depending on the order  $|Z(\inn(G))|$ of $Z(\inn(G))$. If $|Z(\inn(G))|=p^3$, then $|Z_2(G)|\ge p^5$, which is not possible in a nilpotent group of class 4. Next suppose that $|Z(\inn(G))|=p$. Then $G/Z(G)G' \simeq C_{p}$ and thus $|Z(G)G'|=p^{6}$. Now if $|Z(G)|=p^{2}$, then $|G'|=p^{5}$, not possible. If $|Z(G)|=p^{3}$, then $G/Z(G)$ has nilpotence class 3 and order $p^{4}$ and thus $|(G/Z(G))'|=p^{2}$ by \cite[Theorem 2]{kim}, which gives  
$|Z(G)G'|=p^{5}$, a contradiction. If $|Z(G)|=p^{4}$, then $G/Z(G)$ has nilpotence class 3 and order $p^{3}$ and thus$|(G/Z(G))'|=p$ implies that  
$|Z(G)G'|=p^{5}$, a contradiction. Finally suppose that $|Z(\inn(G))|=p^{2}$. In this case $G/Z(G)G' \simeq C_{p} \times C_{p}$ or $C_{p}$ or $C_{p^{2}}$ or $C_{p^3}$ or $C_{p^4}$. Observe that $|Z(G)|=p^2$, because if $|Z(G)|=p^{3}$ or $p^{4}$, then $|Z_{2}(G)|=p^{5}$ or $p^{6}$, which is not possible in a $p$-group of nilpotence class 4.   We claim that $Z(G)\simeq C_{p^2}$ and $|G'|=p^4$. If $Z(G) \simeq C_{p} \times C_{p}$, then  $G/Z(G)G' $ is cyclic and thus $G'$ is not abelian by Lemma 2.1. Since $|G'|<p^5$, $|G'|$ =  $p^{4}$ by \cite[Satz. 7.8, p. 306]{hup}. Therefore $G/Z(G)G' \simeq
C_{p^{2}}$. If $d(G)=3$, then $\exp(G/G')=p$ and if $d(G)=2$, then $G'Z(G)=\Phi(G)$, which is  a contradiction to $\exp(G/G'Z(G))=p^2$. It follows that $Z(G) \simeq C_{p^2}$. If $|G'|=p^2$ or $p^3$, then since $|\cent(G)|=p^2$, $G/G'Z(G)\simeq C_{p^3}$ or $C_{p^4}$, which is not possible by Lemma 2.1 and \cite[Satz. 7.8, p. 306]{hup}. This proves our claim.
\end{proof}

We next  give an example of a group $G$ of order $3^7$ which satisfies equation $(3.1)$. This  group $G$ has ID 131 among the groups of order 2187 in  SmallGroup library of GAP \cite {gap}. Consider
$$
\begin{array}{lcl}
G&=&\langle \alpha_1,\alpha_2,\alpha_3,
\alpha_4,\alpha_5,\alpha_6,\alpha_7
 \;|\;[\alpha_1,\alpha_2]=\alpha_3^{-1},[\alpha_1,\alpha_3]=\alpha_5^{-1},[\alpha_1,\alpha_4]=1,\\
&&\\ 
&& [\alpha_1,\alpha_5]=1,
[\alpha_1,\alpha_6]=1,[\alpha_1,\alpha_7]=1,[\alpha_2,\alpha_3]=\alpha_6^{-1},[\alpha_2,\alpha_4]=1,
[\alpha_2,\alpha_5]=1,\\
&&\\
&&
[\alpha_2,\alpha_6]=1,[\alpha_2,\alpha_7]=1,[\alpha_3,\alpha_4]=1,
[\alpha_3,\alpha_5]=1,[\alpha_3,\alpha_6]=1,[\alpha_3,\alpha_7]=1,\\
&&\\
&&[\alpha_4,\alpha_5]=1,[\alpha_4,\alpha_6]=1,[\alpha_4,\alpha_7]=1,
[\alpha_5,\alpha_6]=1,[\alpha_5,\alpha_7]=1,[\alpha_6,\alpha_7]=1,\\
&&\\
&&
 \alpha_1^{3}=\alpha_4, \alpha_3^{3}=\alpha_7^{-1}, \alpha_4^{3}=\alpha_7,\alpha_1^{27}=\alpha_2^{3}=\alpha_3^{9}=\alpha_4^{9}=
 \alpha_5^{3}=\alpha_6^{3}=\alpha_7^{3}=1\rangle.
 \end{array}
 $$
It can be checked that $Z(G)\simeq \langle\alpha_4\rangle \simeq C_9$, $G'\simeq \langle\alpha_3,\alpha_5,\alpha_6,\alpha_7\rangle$,  $\alpha_4 \in \Phi(G) - G'$, $d(G)=2$, nilpotence class of $G$ is $4$,  $Z(G) \cap G' \simeq \langle \alpha_7 
\rangle  \simeq C_3$ ,   $|G'|=81$,  $Z(G)\notin G'$,  $|\centz(G)|= | \Hom(G/G'Z(G),Z(G))| = | \Hom(C_{p} \times C_{p},C_{p^2})|=p^{2}$  and  $|Z_2(G)/Z(G)|=p^{2}.$

\end{document}